\documentclass[12pt, reqno]{amsart}
\usepackage{mathrsfs}
\usepackage{amsfonts,amsmath,dsfont,amssymb,amsthm,stmaryrd,bbm,color,comment}
\usepackage[hidelinks]{hyperref}
\usepackage{appendix}
\usepackage[pdftex]{graphicx}
\usepackage{bbm}
\usepackage{array}
\newcolumntype{L}{>{\centering\arraybackslash}m{2.5cm}}
\usepackage{multirow}
\usepackage{xcolor}
\usepackage[percent]{overpic}
\usepackage{algorithm}
\usepackage{algpseudocode}
\usepackage{mathtools}

\newtheorem{thm}{Theorem}[section]
\newtheorem{lem}[thm]{Lemma}
\newtheorem{prop}[thm]{Proposition}
\newtheorem{coro}[thm]{Corollary}

\theoremstyle{definition}
\newtheorem{defn}[thm]{Definition}

\theoremstyle{remark}
\newtheorem{rmk}[thm]{Remark}
\numberwithin{equation}{section}

\newcommand{\Var}{\mathrm{Var}}

\newcommand{\R}{\mathbb{R}}

\newcommand{\Z}{\mathbb{Z}}

\newcommand{\E}{\mathbb{E}}

\def\P{{\mathbb P}}

\begin{document}
\setcounter{page}{1}

\color{black}{

\centerline{}

\centerline{}

\title[Symmetrization resistance for exponential and geometric distributions]{Symmetrization resistance for exponential and geometric distributions}

\author[Jiange Li]{Jiange Li} 

\address{Institute for Advanced Study in Mathematics, Harbin Institute of Technology}
\email{jiange.li@hit.edu.cn}






\begin{abstract}

Given a random variable $X$, an independent random variable $Y$ is called a symmetrizer of $X$ if their sum $X+Y$ is symmetric about the origin. The study of symmetrization resistance asks whether every such $Y$ must contain at least as much randomness as $X$. This problem was previously investigated for binary random variables in terms of variance and Shannon entropy.

In this paper, we establish sharp symmetrization resistance results for exponential and geometric distributions. We prove that every independent symmetrizer $Y$ of an exponential random variable $X$ has variance, Rényi and Tsallis entropies of every positive order at least as large as that of $X$. Parallel results are obtained for integer-valued independent symmetrizers of geometric random variables. Equality in each comparison holds precisely when $Y$ is an independent copy of $-X$. Furthermore, we show the majorization of $X$ over $Y$ via establishing sharp concentration inequalities for $Y$. The proofs crucially rely on a differential/difference inversion formula for exponential/geometric convolution, which converts the symmetrization constraint into a hazard-rate inequality and enables us to identity exponential/geometric distributions as extremal distributions of a corresponding optimization problem.

\end{abstract} 

\maketitle


\section{Introduction}

The symmetrization of a random variable is widely used in probability theory and various other fields. One can simply symmetrize a random variable $X$ by adding an independent copy of $-X$. This leads to the following question: How much randomness is necessary to compensate for the asymmetry of $X$? If $X$ is already symmetric about its mean, then a deterministic translation is sufficient to make it symmetric about the origin, requiring no additional randomness. For asymmetric random variables, however, nontrivial randomness is unavoidable. This motivates the study of symmetrization resistance, which quantifies the minimal amount of randomness needed to symmetrize a random variable.


We call a random variable $Y$ an \textit{independent symmetrizer} of a given random variable $X$ if it is independent of $X$ and the distribution of $X+Y$ is symmetric about the origin. A random variable $X$ is called \textit{variance symmetrization resistant} if every
independent symmetrizer $Y$ satisfies
$$
\Var(Y)\geq \Var(X).
$$
Kagan, Mallows, Shepp, Vanderbei and Vardi \cite{KMSVV99} introduced the problem of symmetrization resistance in connection with the Gaussian approximation of distributions by independent additive perturbations subject to a variance constraint. Using linear programming duality, they proved that asymmetric Bernoulli random variables are
variance symmetrization resistant and that the extremal independent symmetrizer is given by an independent copy of $-X$. Later, Pal \cite{Pal08} gave a probabilistic proof of this result based on Skorokhod embedding and stochastic calculus. Non-commutative analogues were recently studied by Chakraborty \cite{Cha25}.

Madiman and Pollard \cite{MP26, Pol23} recently introduced the notion of \textit{entropic symmetrization resistant}, where variance $\Var(\cdot)$ is replaced by Shannon entropy $H(\cdot)$. They proved that any asymmetric Bernoulli random variable $X$ is entropic symmetrization resistant, i.e.,
$$
H(Y)\geq H(X)
$$
for every independent symmetrizer $Y$, and developed a structural description of the space of symmetrizer distributions.

In the current paper, we focus on two fundamental distributions that can be viewed as continuous and discrete counterparts of each other: the exponential and geometric distributions. Our results show that these distributions are resistant to symmetrization under several measures of randomness, including variance, R\'enyi entropy, and Tsallis entropy. These results provide the first examples of continuous and infinite-support distributions exhibiting strong symmetrization resistance.

\begin{thm} [Variance/Entropic symmetrization resistance]\label{thm:main}
Let $Y$ be an independent symmetrizer of $X$. 
\begin{enumerate}
\item If $X\sim\text{Exp}(\lambda)$ is an exponential random variable, then $Y$ is absolutely continuous, and, for all $\alpha>0$, we have
$$
h_\alpha (Y)\ge h_\alpha(X),\quad
T_\alpha (Y)\ge T_\alpha(X), \quad
\Var(Y)\geq \Var(X).
$$
\item If $X\sim\text{Geom}(p)$ is a geometric random variable, and $Y$ is integer-valued, then, for all $\alpha>0$, we have
$$
H_\alpha (Y)\ge H_\alpha(X),\quad
T_\alpha (Y)\ge T_\alpha(X), \quad
\Var(Y)\geq \Var(X).
$$
\end{enumerate}
Here, $h_\alpha(\cdot)$ and $H_\alpha(\cdot)$ denote R\'enyi entropies for continuous and discrete random variables, respectively, and $T_\alpha(\cdot)$ denotes Tsallis entropy. Equality in each comparison holds if and only if $Y$ is an independent copy of $-X$. 
\end{thm}

Our proofs of these comparison bounds follow from a common approach that crucially relies on special features of exponential and geometric convolutions. Let $f_Y$ and $f_Z$ denote densities or probability mass functions of $Y$ and $Z=X+Y$, respectively. In the exponential case, the convolution identity can be inverted as
$$
f_Y=f_Z+\lambda^{-1}f_Z'.
$$
The positivity of $f_Y$ imposes the following differential constraint 
$$
f_Z+\lambda^{-1}f_Z'\ge 0.
$$
This constraint can be interpreted as a
bound on the hazard rate of $|Z|$, which enables us to identify exponential distribution as the extremizer of a corresponding minimization problem. Comparison results for the geometric distribution follow from the discrete counterpart of the previous inversion formula
$$
pf_Y(k)=f_Z(k)-qf_Z(k-1).
$$
This identity converts the symmetrization condition into a discrete hazard-rate inequality. Then a tail comparison argument shows that geometric distribution is extremal among all distributions satisfying this constraint.

Furthermore, we establish the following sharp concentration bounds, which imply the majorization of $X$ over $Y$. Therefore, from a different perspective, this shows that $Y$ is more random than $X$.

\begin{thm} [Concentration bounds]
Let $Y$ be an independent symmetrizer of $X$. 
\begin{enumerate}
\item If $X\sim\text{Exp}(\lambda)$ is an exponential random variable, then for every Borel measurable set $A\subset\R$ of finite measure $|A|$, we have
$$
 \P(Y\in A)\le 1-e^{-\lambda|A|}.
$$
\item If $X\sim\text{Geom}(p)$ is a geometric random variable, and $Y$ is integer-valued, then for every finite subset $A\subset\Z$, we have
$$
 \P(Y\in A)\le 1-q^{|A|}.
$$
\end{enumerate}
Both equalities can be achieved by an independent copy of $-X$ and suitable sets $A$. 
\end{thm}

\section{Exponential symmetrization}

\begin{defn}[Continuous R\'enyi and Tsallis entropies]
Let $X$ be a real-valued random variable with density $f$ with respect to the Lebesgue measure. For $\alpha\in (0, 1)\cup (1, \infty)$, the R\'enyi entropy $h_\alpha(X)$ of order $\alpha$ is defined by
$$
h_\alpha(X) = \frac{1}{1-\alpha}\log \int_{\mathbb{R}}f(x)^\alpha dx
$$
and the Tsallis entropy $T_\alpha(X)$ of order $\alpha$ is defined by
$$
T_\alpha(X) = \frac{1}{1-\alpha}\left(\int_{\mathbb{R}}f(x)^\alpha dx-1\right).
$$
Taking the limit $\alpha\to 1$, both entropies are reduced to the Boltzman-Shannon entropy 
$$
h(X)=-\int_{\mathbb{R}}f(x)\log f(x)dx.
$$
\end{defn}

\begin{rmk}\label{rmk:scaling}
One can check the following scaling properties of R\'enyi and Tsallis entropies for real-valued random variables. For $\lambda\ne 0$, we have
\begin{align*}
h_\alpha(\lambda X) &= h_\alpha(X)+\log |\lambda|,\\
T_\alpha(\lambda X) &= |\lambda|^{1-\alpha}T_\alpha(X)+\frac{|\lambda|^{1-\alpha}-1}{1-\alpha}.
\end{align*}
\end{rmk}

\begin{prop} [Inversion formula] \label{prop:inversion}
Let $\lambda>0$. Let $X\sim \text{Exp}(\lambda)$ be an exponential random variable with density $f_X(x)=\lambda e^{-\lambda x}\mathbbm{1}_{\mathbb{R}_+}(x)$. Let $Y$ be an independent symmetrizer of $X$. Write $Z=X+Y$. Then $Y$ is absolutely continuous with respect to the Lebesgue measure. Furthermore, for almost all $z\in\R$, we have
\begin{equation}\label{eq:pdf-y}
f_Y(z)=f_Z(z)+\lambda^{-1}f_Z'(z),
\end{equation}
where $f_Y$ and $f_Z$ are densities of $Y$ and $Z$, respectively.
\end{prop}

\begin{proof}
Let $\mu_X, \mu_Y, \mu_{-Y}, \mu_Z$ denote the distributions of $X, Y, -Y, Z$, respectively. One can check that the distributional derivative of $\mu_X$ satisfies
$$
D\mu_X=\lambda\delta_0-\lambda\mu_X,
$$
where $\delta_0$ is the Dirac measure. Therefore, we have
$$
D\mu_Z=D(\mu_X\ast\mu_Y)=(D\mu_X)\ast\mu_Y=\lambda \mu_Y-\lambda\mu_Z,
$$
which gives
\begin{equation}\label{eq:inversion}
\mu_Y=\mu_Z+\lambda^{-1}D\mu_Z.
\end{equation}
Since $Z$ is symmetric, the signed measure $D\mu_Z$ satisfies $(D\mu_Z)(-A)=-(D\mu_Z)(A)$ for all Borel measurable sets $A\subset \R$. Then we obtain from identity \eqref{eq:inversion} that
$$
\mu_{-Y}=\mu_Z-\lambda^{-1}D\mu_Z.
$$
This, combined with \eqref{eq:inversion}, gives
$$
\mu_Y+\mu_{-Y}=2\mu_Z.
$$
Since $\mu_{-Y}$ is a positive measure, we obtain $\mu_Y\ll \mu_Z$. Clearly, $Z$ is absolutely continuous, and hence, $Y$ is also absolutely continuous. By Lebesgue's differentiation theorem, we know that $f_Z$ is differentiable almost everywhere. Hence, we have $D\mu_Z=f_Z'(z)dz$. Combined with \eqref{eq:inversion}, we obtain the inversion formula \eqref{eq:pdf-y}.
\end{proof}

\begin{coro}
Under notations of Proposition \ref{prop:inversion}, we have 
\begin{equation}\label{eq:maximum}
\|f_Y\|_\infty\le 2\|f_Z\|_\infty\le \lambda.
\end{equation}
Furthermore, equation $\|f_Y\|_\infty=\lambda$ holds if and only if $Y$ is an independent copy of $-X$.
\end{coro}

\begin{proof}
Since $Z$ is symmetric, we have $f_Z'(-z)=-f_Z'(z)$ for almost all $z\in\R$. By inversion formula \eqref{eq:inversion}, we obtain 
$$
f_Y(-z)=f_Z(z)-\lambda^{-1}f_Z'(z).
$$
Combined with \eqref{eq:inversion} and the positivity of $f_Y$, this implies for almost all $z\in\R$ that
\begin{equation}\label{eq:diff-ineq}
|f_Z'(z)|\le \lambda f_Z(z).
\end{equation}
Using the inversion formula \eqref{eq:inversion} again, we obtain for almost all $z\in\R$ that
$$
f_Y(z)\le 2f_Z(z).
$$
This proves the first part of \eqref{eq:maximum}. We apply Gronwall's inequality to the differential inequality \eqref{eq:diff-ineq} and obtain for all $z_0, z\in \R$ that
\begin{equation}\label{eq:gronwall}
f_Z(z)\ge f(z_0)e^{-\lambda|z-z_0|}.
\end{equation}
Integrate both sides over $z$ to obtain
$$
1=\int_{\R}f_Z(z)dz\ge \int_{\R}f(z_0)e^{-\lambda|z-z_0|}dz=2\lambda^{-1}f(z_0).
$$
Since $z_0$ is arbitrary, this proves the second part of \eqref{eq:maximum}.

\textbf{Equality case}. By \eqref{eq:maximum}, equation $\|f_Y\|_\infty=\lambda$ gives $\|f_Z\|_\infty=\lambda/2$. Continuity and decay of $f_Z$ imply that $f(z_0)=\lambda/2$ for some $z_0\in\R$. Then equality in \eqref{eq:gronwall} has to hold. Symmetry of $Z$ forces $z_0=0$. Then we can conclude from the inversion formula \eqref{eq:inversion} that $Y\overset{d}{=}-X$.
\end{proof}

\begin{thm}\label{thm:symm-exp-ent}
Let $\lambda>0$. Let $X\sim \text{Exp}(\lambda)$ be an exponential random variable with the probability density function $f_X(x)=\lambda e^{-\lambda x}\mathbbm{1}_{\mathbb{R}_+}(x)$. Let $Y$ be an independent symmetrizer of $X$. For all $\alpha>0$, we have
\begin{equation}\label{eq:exp-renyi-y-x}
h_\alpha (Y)\ge h_\alpha(X)
\end{equation}
and
\begin{equation}\label{eq:exp-tsallis-y-x}
T_\alpha (Y)\ge T_\alpha(X).
\end{equation}
Equality holds if and only if $Y$ is an independent copy of $-X$.
\end{thm}

\begin{proof}
By the scaling property in Remark \ref{rmk:scaling}, we can assume that $\lambda=1$. We first prove the statement for the case $\alpha\neq 1$; the case $\alpha=1$ follows from a similar argument. 

\textbf{Case 1: $\alpha\ne 1$}. We denote by $f_Y$ the probability density function of $Y$. 
By definitions of R\'enyi entropy and Tsallis entropy, both \eqref{eq:exp-renyi-y-x} and \eqref{eq:exp-tsallis-y-x} for $\alpha\neq 1$ are equivalent to
\begin{align}
\int_{\mathbb{R}}f_Y(z)^\alpha dz &\ge \int_0^\infty f_X(z)^\alpha dz\quad \text{if}~0<\alpha<1, \label{eq:exp-y-x-1}\\
\int_{\mathbb{R}}f_Y(z)^\alpha dz &\le \int_0^\infty f_X(z)^\alpha dz\quad \text{if}~\alpha>1. \label{eq:exp-y-x-2}
\end{align}
The $\lambda=1$ case of the inversion formula \eqref{eq:pdf-y} gives
\begin{equation}\label{eq:pdf-y-a}
f_Y(z)=f_Z(z)+f_Z'(z).
\end{equation}
Write $u(z)=f_Z'(z)/f_Z(z)$ for $z\in\mathbb{R}$. (We have $f_Z(z)>0$ for all $z\in\mathbb{R}$. To see this, observe that the symmetry of $Z$ implies that $\mathbb{P}(Y\le 0)>0$. This fact and the following convolution formula 
$$
f_Z(z)=\int_{-\infty}^z f_X(z-y)f_Y(y)dy=e^{-z}\int_{-\infty}^z f_Y(y)e^ydy.
$$
yield that $f_Z(z)>0$ for all $z\ge 0$. Using the symmetry of $Z$ again, we obtain the positivity of $f_Z$ on the whole real line.) Since $f_Z$ is symmetric, $u(z)$ is an odd function and we have
\begin{align}\label{eq:ent-y}
\int_{\mathbb{R}}f_Y(z)^\alpha d x &=\int_{\mathbb{R}}f_Z(z)^\alpha (1+u(z))^\alpha dz \notag\\
&=\int_0^\infty f_Z(z)^\alpha\left((1+u(z))^\alpha+(1-u(z))^\alpha
\right)dz.
\end{align}
The non-negativity of $f_Y$ and equation \eqref{eq:pdf-y-a} imply that $|u(z)|\le 1$ for almost all $z\in\mathbb{R}$. For all $|t|\le 1$, one can check that
\begin{align*}
(1+t)^\alpha+(1-t)^\alpha & \ge 2^\alpha\quad \text{if}~0<\alpha<1,\\
(1+t)^\alpha+(1-t)^\alpha & \le 2^\alpha\quad \text{if}~\alpha>1.
\end{align*}
Combining this and equation \eqref{eq:ent-y}, we obtain 
\begin{align}
\int_{\mathbb{R}}f_Y(z)^\alpha dz &\ge \int_0^\infty (2f_Z(z))^\alpha dz \quad \text{if}~0<\alpha<1, \label{eq:exp-fy-phi-1}\\
\int_{\mathbb{R}}f_Y(z)^\alpha dz &\le \int_0^\infty (2f_Z(z))^\alpha dz \quad \text{if}~\alpha>1. \label{eq:exp-fy-phi-2}
\end{align}
Write $\phi(z)=2f_Z(z)$ for $z\ge 0$, which is the density of $|Z|$. Again, using the non-negativity of $f_Y$ and equation \eqref{eq:pdf-y-a}, we have for  almost all $z\ge 0$ that
$$
\phi(z)+\phi'(z)\ge 0
$$
which implies that
\begin{align}\label{eq:phi-bound}
\phi(z)=-\int_z^\infty \phi'(x)dx\le \int_z^\infty \phi(x)dx=1-F_{|Z|}(z),
\end{align}
where $F_{|Z|}$ is the cumulative distribution function of $|Z|$. Therefore we have
\begin{align}
\int_0^\infty \phi(z)^\alpha dz &\ge \int_0^\infty \phi(z)(1-F_{|Z|}(z))^{\alpha-1} dz \quad \text{if}~0<\alpha<1, \label{eq:exp-phi-fx-1}\\
\int_0^\infty \phi(z)^\alpha dz &\le \int_0^\infty \phi(z)(1-F_{|Z|}(z))^{\alpha-1} dz \quad \text{if}~\alpha>1. \label{eq:exp-phi-fx-2}
\end{align}
Note that
$$
\int_0^\infty \phi(z)(1-F_{|Z|}(z))^{\alpha-1} dz = \frac{1}{\alpha}=\int_0^\infty f_X(z)^\alpha dz.
$$
Put this together with \eqref{eq:exp-fy-phi-1}, \eqref{eq:exp-fy-phi-2}, \eqref{eq:exp-phi-fx-1}, \eqref{eq:exp-phi-fx-2}. We obtain inequalities \eqref{eq:exp-y-x-1} and \eqref{eq:exp-y-x-2}. 

\textbf{Case 2: $\alpha= 1$}. In this case, the task is to show that
\begin{align}\label{eq:exp-y-x-3}
\int_{\mathbb{R}}f_Y(z)\log f_Y(z) dz &\le \int_0^\infty f_X(z)\log f_X(z) dz.
\end{align}
Substituting $f_Y(z)=f_Z(z)(1+u(z))$ and using the property that $u(z)$ is odd, one can check that
\begin{align}\label{eq:exp-fy-phi-3}
\int_{\mathbb{R}}f_Y(z)\log f_Y(z) dz & =2\int_0^\infty f_Z(z)\log f_Z(z)dz \notag\\
&\quad+\int_0^\infty f_Z(z)[(1+u(z))\log (1+u(z))+(1-u(z))\log (1-u(z))]dz \notag\\
&=\int_0^\infty \phi_Z(z)\log \phi_Z(z)dz-\log 2 \notag\\
&\quad+\int_0^\infty f_Z(z)[(1+u(z))\log (1+u(z))+(1-u(z))\log (1-u(z))]dz \notag\\
&\le \int_0^\infty \phi_Z(z)\log \phi_Z(z)dz.
\end{align}
In the last inequality, we use the following inequality
$$
(1+t)\log(1+t)+(1-t)\log(1-t)\le 2\log 2 \quad \text{for}~|t|\le 1.
$$
Invoke inequality \eqref{eq:phi-bound}. Then we obtain
\begin{align}\label{eq:exp-phi-fx-3}
\int_0^\infty \phi_Z(z)\log \phi_Z(z)dz &\le \int_0^\infty \phi_Z(z)\log (1-F_{|Z|}(z))dz\\
&=-1=\int_0^\infty f_X(z)\log f_X(z)dz \notag.
\end{align}
This, together with \eqref{eq:exp-fy-phi-3}, gives the desired inequality \eqref{eq:exp-y-x-3}.

\textbf{Equality case.} The following statement is always in the almost everywhere sense. Equalities in \eqref{eq:exp-y-x-1}, \eqref{eq:exp-y-x-2} and \eqref{eq:exp-y-x-3} force equalities in \eqref{eq:exp-phi-fx-1}, \eqref{eq:exp-phi-fx-2} and \eqref{eq:exp-phi-fx-3} to hold. Then we must have $\phi(z)=1-F_{|Z|}(z)$, which yields $\phi(z)+\phi'(z)=0$ for $z\ge 0$. Thus we have $\phi(z)=e^{-z}$ for $z\ge 0$ and therefore $f_Z(z)=\frac{1}{2}e^{-|z|}$. This, together with the inversion formula \eqref{eq:pdf-y}, gives $f_Y(z)=e^z\mathbbm{1}_{\mathbb{R}_-}(z)$. Hence, $Y$ has the same distribution as $-X$.
\end{proof}

\begin{rmk}
The exponential distribution enables the deconvolution that represents the density of $Y$ in terms of the density of $Z$ (see equation $\eqref{eq:pdf-y}$). The central inequality is
$$
\phi(z)+\phi'(z)\ge 0.
$$
We prove that, subject to this inequality, exponential minimizes R\'enyi entropies over all probability densities on $\mathbb{R}_+$. Also, observe that the above inequality yields that
$$
r(z):=\frac{\phi(z)}{1-F(z)}\le 1.
$$
Hence, among distributions with hazard rate bounded by 1, the exponential minimizes every R\'enyi entropy.
\end{rmk}

Let $U$ and $V$ be real-valued random variables. We say that $V$ has \textit{first-order stochastic dominance} over $U$ if $\P(U\ge t)\le \P(V\ge t)$ for all $t\in\R$, which is equivalent to that $\E f(U)\le \E f(V)$ for all non-decreasing functions $f$.

\begin{thm}
Let $\lambda>0$. Let $X\sim \text{Exp}(\lambda)$ be an exponential random variable with the probability density function $f_X(x)=\lambda e^{-\lambda x}\mathbbm{1}_{\mathbb{R}_+}(x)$. Let $Y$ be an independent symmetrizer of $X$.  Then we have
$$
\Var (Y)\ge \Var(X).
$$
Equality holds if and only if $Y$ is an independent copy of $-X$.
\end{thm}

\begin{proof}
The theorem trivially holds if $\Var(Y)=\infty$. Thus we assume that $\Var(Y)<\infty$. We again assume $\lambda=1$, write $Z=X+Y$, and denote by $\phi(z)$ and $F_{|Z|}(z)$ the probability density function and the cumulative distribution function of $|Z|$, respectively. We have shown in the proof of Theorem \ref{thm:symm-exp-ent} (inequality \eqref{eq:phi-bound}) for almost all $z\ge 0$ that
$$
\phi(z)\le 1-F_{|Z|}(z),
$$ 
which implies that
$$
\frac{d}{dz}\log (1-F_{|Z|}(z))=-\frac{\phi(z)}{1-F_{|Z|}(z)}\ge -1.
$$
Therefore we have for all $z\ge 0$ that
$$
\P(|Z|\ge z)=1-F_{|Z|}(z)\ge e^{-z}=\P(X\ge z).
$$
This means that $|Z|$ has first-order stochastic dominance over $X$. Since $z^2$ is increasing, we have
$$
\Var (Z)=\E Z^2=\E|Z|^2\ge \E X^2=2. 
$$
Then we obtain 
$$
\Var (Y)=\Var(Z)-\Var (X)\ge 1=\Var(X).
$$

\textbf{Equality case}. As the proof reveals, if the identity $\Var (Y)=\Var(X)$ holds, it is necessary to have for almost all $z\ge 0$ that $\phi(z)=1-F_{|Z|}(z)$. This gives the same conclusion as discussed in the equality case of Theorem \ref{thm:symm-exp-ent}. Hence, $Y$ must have the same distribution as $-X$.
\end{proof}

We next establish a sharp concentration bound for $Y$, which implies the majorization of $X$ over $Y$. We first prove the following lemma.

\begin{lem}\label{lem:hazard ineq}
Let $\lambda>0$. Let $\phi$ be a probability density on $[0,\infty)$. Suppose that
$$
\phi(t)\leq\lambda \int_t^\infty\phi(s)\,ds \quad\text{for almost every }t\geq0.
$$
Then, for every Borel measurable set $B\subset[0,\infty)$ of finite measure $|B|$, we have
$$
\int_B\phi(t)\,dt\leq 1-e^{-\lambda|B|}.
$$
\end{lem}

\begin{proof}
Set $S(t)=\int_t^\infty\phi(s)\,ds$. Define
$$
U(t)=1-\int_0^t\mathbf1_B(s)\phi(s)\,ds,
\qquad
L(t)=\int_0^t\mathbf1_B(s)\,ds.
$$
It is clear that
$$
U(t)\geq1-\int_0^t\phi(s)\,ds=S(t).
$$
Consequently, for almost every $t\geq0$, we have
$$
\frac{d}{dt}\bigl(e^{\lambda L(t)}\cdot U(t)\bigr)=e^{\lambda L(t)}\mathbf1_B(t)\bigl(\lambda U(t)-\phi(t)\bigr)\geq0.
$$
Since $U(0)=1$ and $L(0)=0$, we obtain
$$
U(t)\geq e^{-\lambda L(t)}.
$$
Let $t\to\infty$, and use $|B|<\infty$ to obtain
$$
1-\int_B\phi(s)\,ds\geq e^{-\lambda|B|}.
$$
This is the desired inequality.
\end{proof}

\begin{thm}\label{thm:concentration-exp}
Let $\lambda>0$. Let $X\sim \text{Exp}(\lambda)$ be an exponential random variable with the probability density function $f_X(x)=\lambda e^{-\lambda x}\mathbbm{1}_{\mathbb{R}_+}(x)$. Let $Y$ be an independent symmetrizer of $X$.  For every Borel measurable set $A\subset \R$ of finite measure $|A|$, we have
\begin{equation}\label{eq:concentration-exp}
\P(Y\in A)\le 1-e^{-\lambda |A|}.
\end{equation}
Equality can be achieved by $Y\overset{d}{=}-X$ and $A=[-|A|, 0]$.
\end{thm}

\begin{proof}
We have established in Proposition \ref{prop:inversion} the following inversion formula 
\begin{equation}\label{eq:inversion-conc}
f_Y(z)=f_Z(z)+\lambda^{-1}f_Z'(z)
\end{equation}
for almost all $z\in\R$. The positivity of $f_Y$ yields 
$$
|f_Z'(z)|\le \lambda f_Z(z), \quad f_Y(z)\le 2f_Z(z)
$$ 
for almost all $z\in\R$. Decompose $A=A_+\cup A_-$, where $A_+=A\cap [0, \infty)$ and $A_-=A\cap (-\infty, 0)$. We have
$$
\P(Y\in A) = \left(\int_{A_+}+\int_{A_-}\right)(f_Z(z)+\lambda^{-1}f'(z))dz.
$$
Since $Z$ is symmetric, we have $f'(-z)=-f_Z'(z)$ for almost all $z\in\R$. Then we have
$$
\int_{A_-}(f_Z(z)+\lambda^{-1}f'(z))dz=\int_{-A_-}(f_Z(z)-\lambda^{-1}f'(z))dz.
$$
Putting them together to obtain
\begin{align}\label{eq:prob y-in-a}
\P(Y\in A) &=\int_{A_+\setminus -A_-}(f_Z(z)+\lambda^{-1}f'(z))dz+2\int_{A_+\cap(-A_-)}f(z)dz  \notag\\
&\quad +\int_{ -A_-\setminus A_+}(f_Z(z)-\lambda^{-1}f'(z))dz\notag\\
&\le \int_{A_+\cup (-A_-)}(2f_Z(z))dz\notag\\
&=\int_B\phi(z)dz.
\end{align}
Here, $B=A_+\cup (-A_-)$ and $\phi(z)=2f_Z(z)$ for $z\ge 0$. Clearly, $|B|\le |A_+|+|-A_-|=|A|$.
The inversion formula \eqref{eq:inversion-conc} and the positivity of $f_Y$ gives $\phi(z)+\lambda^{-1}\phi'(z)\ge0$ for almost all $z\in\R$,
from which we obtain
$$
\phi(z)=-\int_z^\infty \phi'(x)dx\le \lambda \int_z^\infty \phi(x)dx.
$$
Then we apply Lemma \ref{lem:hazard ineq} to obtain 
$$
\int_B\phi(z)dz\le 1-e^{-\lambda |B|}\le 1-e^{-\lambda |A|}.
$$
Combined with \eqref{eq:prob y-in-a}, this yields the desired result.
\end{proof}

\begin{rmk}\label{rmk:concentration}
Concentration inequality \eqref{eq:concentration-exp} can be interpreted as the majorization of $X$ over $Y$. For two integrable functions $f, g: \R\to\R_+$ such that $\int_{\R}f(x)dx=\int_{\R}g(x)dx$, we say that $f$ is \textit{majorized} by $g$ if it holds for all $t\ge 0$ that
$$
\int_{\R}(f(x)-t)_+dx\le \int_{\R}(g(x)-t)_+dx.
$$
A sufficient condition for the majorization of $g$ over $f$ given in Proposition 2.4 of \cite{FL25} is that, for any Borel measurable set $A$ of finite measure, there exists a Borel measurable set $B$ with equal measure as $A$ such that
$$
\int_Af(x)dx\le \int_Bg(x)dx.
$$
Let $f_X$ and $f_Y$ denote the probability density functions of $X$ and $Y$, respectively. Inequality \eqref{eq:concentration-exp} can then be rewritten as, for every Borel measurable set $A\subset \R$ of finite measure $|A|$, we have
$$
\int_A f_Y(z)dz\le \int_0^{|A|}f_X(z)dz.
$$
Thus we have the majorization of $f_X$ over $f_Y$. By Proposition 7.3 of \cite{WM14}, for every convex function $\varphi: \R_+\to\R$ with $\varphi(0)=0$ and continuous at 0, it holds that
$$
\int_{\R}\varphi(f_Y(z))dz\le \int_{\R}\varphi(f_X(z))dz.
$$
This implies the entropic comparisons \eqref{eq:exp-renyi-y-x} and \eqref{eq:exp-tsallis-y-x} by taking
$$
\varphi(t)=
\begin{cases}
t^\alpha & \text{for}~\alpha>1\\
t\log t & \text{for}~\alpha=1\\
-t^\alpha & \text{for}~\alpha<1.
\end{cases}
$$
\end{rmk}


\section{Geometric symmetrization}

\begin{defn}[Discrete R\'enyi and Tsallis entropies]
Let $X$ be a discrete random variable with the probability mass function $\{f(k)\}_{k\in\mathbb{Z}}$. For $\alpha\in (0, 1)\cup (1, \infty)$, the R\'enyi entropy $H_\alpha(X)$ of order $\alpha$ is defined by
$$
H_\alpha(X) = \frac{1}{1-\alpha}\log \sum_{k\in\mathbb{Z}}f(k)^\alpha
$$
and the Tsallis entropy $T_\alpha(X)$ of order $\alpha$ is defined by
$$
T_\alpha(X) = \frac{1}{1-\alpha}\left(\sum_{k\in\mathbb{Z}}f(k)^\alpha-1\right).
$$
Taking the limit $\alpha\to 1$, both entropies are reduced to Shannon entropy 
$$
H(X)=-\sum_{k\in\mathbb{Z}}f(k)\log f(k).
$$ 
\end{defn}

\begin{thm}\label{thm:symm-geo-ent}
Let $0<p<1$ and write $q:=1-p$. Let $X\sim\text{Geom}(p)$ be a geometric random variable with the probability mass function $f_X(k)=pq^k$ for $k\in \mathbb{Z}_{\ge 0}$. Let $Y$ be an integer-valued independent symmetrizer of $X$. Then we have for all $\alpha>0$ that
\begin{equation}\label{eq:geo-renyi-y-x}
H_\alpha (Y)\ge H_\alpha(X)
\end{equation}
and
\begin{equation}\label{eq:geo-tsallis-y-x}
T_\alpha (Y)\ge T_\alpha(X).
\end{equation}
Equality holds if and only if $Y$ is an independent copy of $-X$.
\end{thm}

\begin{proof}
We first prove the statement for the case $\alpha\neq 1$; the case $\alpha=1$ follows from a similar argument. 

\textbf{Case 1: $\alpha\ne 1$}. Let $f_Y$ be the probability mass function of $Y$. By definitions of R\'enyi and Tsallis entropies, both \eqref{eq:geo-renyi-y-x} and \eqref{eq:geo-tsallis-y-x} for $\alpha\neq 1$ are equivalent to
\begin{align}
\sum_{k\in \mathbb{Z}} f_Y(k)^\alpha &\ge \sum_{k=0}^\infty f_X(k)^\alpha\quad \text{if}~0<\alpha<1, \label{eq:geo-renyi-y-x-1}\\
\sum_{k\in \mathbb{Z}} f_Y(k)^\alpha &\le \sum_{k=0}^\infty f_X(k)^\alpha\quad \text{if}~\alpha>1. \label{eq:geo-tsallis-y-x-1}
\end{align}
Below we prove \eqref{eq:geo-renyi-y-x-1} and \eqref{eq:geo-tsallis-y-x-1} by introducing a probability mass function $\{\phi(k)\}_{k=1}^\infty$ whose moment can be sandwiched between the corresponding moments of $f_X$ and $f_Y$. 

We write $Z:=X+Y$ and denote by $f_Z$ the probability mass function of $Z$. We define
\begin{equation}\label{eq:phi}
\phi(k):=f_Z(k)+f_Z(k-1).
\end{equation} 
By the symmetry $f_Z(k)=f_Z(-k)$, we have
$$
\sum_{k=1}^\infty \phi(k)=f_Z(0)+2\sum_{k=1}^\infty f_Z(k)=\sum_{k\in\Z}f_Z(k)=1.
$$
Thus $\{\phi(k)\}_{k=1}^\infty$ is indeed a probability mass function. We next show the identity
\begin{equation}\label{eq:phi-1}
\phi(k)=f_Y(k)+f_Y(1-k).
\end{equation}
By definition, we have 
$$
f_Z(k)=\sum_{j=0}^\infty f_X(j)f_Y(k-j)=\sum_{j=0}^\infty pq^jf_Y(k-j),
$$
from which we obtain the inversion formula
\begin{equation}\label{eq:pmf-y}
pf_Y(k)=f_Z(k)-qf_Z(k-1).
\end{equation}
Using the symmetry of $f_Z$, we obtain from \eqref{eq:pmf-y} that
$$
pf_Y(1-k)=f_Z(k-1)-qf_Z(k).
$$
Combine these two identities to obtain
$$
f_Y(k)+f_Y(1-k)=f_Z(k)+f_Z(k-1).
$$
This, together with the definition of $\phi$ in \eqref{eq:phi}, gives identity \eqref{eq:phi-1}. 

Using identity \eqref{eq:phi-1} and the concavity of $t^\alpha$ for $0<\alpha<1$ (respectively, convexity for $\alpha>1$), we have
\begin{align*}
f_Y(k)^\alpha+f_Y(1-k)^\alpha & \ge \phi(k)^\alpha\quad \text{if}~0<\alpha<1,\\
f_Y(k)^\alpha+f_Y(1-k)^\alpha & \le \phi(k)^\alpha\quad \text{if}~\alpha>1.
\end{align*}
Therefore, we have for $0<\alpha<1$ that
\begin{equation}\label{eq:moment-f_y-phi-1}
\sum_{k\in\mathbb{Z}}f_Y(k)^\alpha =\sum_{k=1}^\infty (f_Y(k)^\alpha+f_Y(1-k)^\alpha)\ge \sum_{k=1}^\infty \phi(k)^\alpha,
\end{equation}
and for $\alpha>1$ that
\begin{equation}\label{eq:moment-f_y-phi-2}
\sum_{k\in\mathbb{Z}}f_Y(k)^\alpha\le \sum_{k=1}^\infty \phi(k)^\alpha.
\end{equation}
Therefore, for $0<\alpha<1$ (respectively, $\alpha>1$), the $\alpha$-th moment of $f_Y$ is bounded below (respectively, above) by the corresponding moment of $\phi$.

We next establish an analogous relationship between $f_X$ and $\phi$. Write $\phi(k)=R_k-R_{k+1}$, where the tail probability $R_k$ is defined as
$$
R_k:=\sum_{j=k}^\infty \phi(j).
$$
By the non-negativity of $f_Y$ and equation \eqref{eq:pmf-y}, we have
$$
f_Z(k)\ge qf_Z(k-1),
$$
which, together with \eqref{eq:phi}, gives
\begin{equation}\label{eq:phi-k-k+1}
\phi(k+1)\ge q\phi(k).
\end{equation}
This implies the inequality
\begin{equation}\label{eq:T-k-k+1}
R_{k+1}\ge qR_k.
\end{equation}
One can check that $(1-t)^\alpha/(1-t^\alpha)$ is an increasing (respectively, decreasing) function on $[0, 1)$ for $0<\alpha<1$ (respectively, $\alpha>1$). Write $t_k:=R_{k+1}/R_k\in [q, 1]$. For $0<\alpha<1$, we apply inequality \eqref{eq:T-k-k+1} to obtain
\begin{align*}
\phi(k)^\alpha &=R_k^\alpha(1-t_k)^\alpha=\frac{(1-t_k)^\alpha}{1-t_k^\alpha}\cdot R_k^\alpha(1-t_k^\alpha)\\
&\ge \frac{(1-q)^\alpha}{1-q^\alpha}(R_k^\alpha-R_{k+1}^\alpha)\\
&=\frac{p^\alpha}{1-q^\alpha}(R_k^\alpha-R_{k+1}^\alpha).
\end{align*}
Note $R_1=1$. Sum up over $k\ge 1$ and obtain
$$
\sum_{k=1}^\infty \phi(k)^\alpha\ge \frac{p^\alpha}{1-q^\alpha}=\sum_{k=0}^\infty f_X(k)^\alpha.
$$
For $\alpha>1$, the two above inequalities are reversed and we obtain
$$
\sum_{k=1}^\infty \phi(k)^\alpha\le \sum_{k=0}^\infty f_X(k)^\alpha.
$$
Then, together with \eqref{eq:moment-f_y-phi-1} and \eqref{eq:moment-f_y-phi-2}, we obtain \eqref{eq:geo-renyi-y-x-1} and \eqref{eq:geo-tsallis-y-x-1}. 

\textbf{Case 2: $\alpha=1$}. In this case, the task is to show that
\begin{equation}\label{eq:geo-renyi-y-x-3}
\sum_{k\in \mathbb{Z}} f_Y(k)\log f_Y(k) \le \sum_{k=0}^\infty f_X(k)\log f_X(k).
\end{equation}
For the left hand side, we have
\begin{align}\label{eq:moment-f_y-phi-3}
\sum_{k\in \mathbb{Z}} f_Y(k)\log f_Y(k) &=\sum_{k=1}^\infty [f_Y(k)\log f_Y(k)+f_Y(1-k)\log f_Y(1-k)] \notag\\
&\le \sum_{k=1}^\infty\phi(k)\log \phi(k).
\end{align}
The inequality follows from identity \eqref{eq:phi-1} and that, for $a, b\ge0, a+b>0$, we have
$$
a\log\frac{a}{a+b}+b\log\frac{a}{a+b}\le 0.
$$
Substituting $\phi(k)=R_k-R_{k+1}$ and $t_k=R_{k+1}/R_k$, one can verify that
$$
\phi(k)\log \phi(k)=\phi(k)\left(\log(1-t_k)+\frac{t_k\log t_k}{1-t_k}\right)+R_k\log R_k-R_{k+1}\log R_{k+1}.
$$
One can check that the function $\log(1-t)+(t\log t)/(1-t)$ is non-increasing on $(0, 1)$. Since $t_k\in [q, 1]$, $R_k\to 0$ as $k\to\infty$, and $R_1=1$, we sum the above identity over $k$ to obtain
\begin{align*}
\sum_{k=1}^\infty\phi(k)\log \phi(k) &\le \left(\log (1-q)+\frac{q\log q}{1-q}\right)\sum_{k=1}^\infty \phi(k)\\
&=\log p+\frac{q}{p}\log q\\
&=\sum_{k=0}^\infty f_X(k)\log f_X(k).
\end{align*}
This, together with \eqref{eq:moment-f_y-phi-3}, yields \eqref{eq:geo-renyi-y-x-3}.

\textbf{Equality case.} As our proof reveals, equalities in \eqref{eq:geo-renyi-y-x} and \eqref{eq:geo-tsallis-y-x} 
 forces the equation
$R_{k+1}=qR_k$, from which we obtain $R_k=q^{k-1}$ and therefore 
$\phi(k)=pq^{k-1}$. Then we solve equation \eqref{eq:phi} and obtain
$$
f_Z(k)=\frac{pq^{|k|}}{1+q},\quad k\in\mathbb{Z}.
$$
This, together with the inversion formula \eqref{eq:pmf-y}, gives 
$$
f_Y(k)=
\begin{cases}
0 & \text{for}~k\ge 1,\\
pq^{|k|} & \text{for}~k\le 0.
\end{cases}
$$
Therefore $Y$ has the same distribution as $-X$. 
\end{proof}

\begin{rmk}
The inversion formula of $f_Y$ in \eqref{eq:pmf-y} can be rewritten as
$$
f_Y(k)=f_Z(k)+\frac{q}{p}(f_Z(k)-f_Z(k-1)),
$$
which is the discrete analogy of equation \eqref{eq:pdf-y}. We can apply the symmetry of $f_Z$ to obtain
\begin{align*}
\sum_{k\in\mathbb{Z}}f_Y(k)^\alpha &=\sum_{k=1}^\infty f_Z(k)^\alpha\Bigg\{ 
\left[1+\frac{q}{p}\left(1-\frac{f_Z(k-1)}{f_Z(k)}\right)\right]^\alpha+ \left[1-\frac{1}{p}\left(1-\frac{f_Z(k-1)}{f_Z(k)}\right)\right]^\alpha \Bigg\}.
\end{align*}
By equation \eqref{eq:pmf-y} and the non-negativity of $f_Y$, we have $f_Z(k)\ge qf_Z(k-1)$. Since $Z$ is symmetric, we flip $k$ to $-k$ and obtain $f_Z(k)\ge qf_Z(k+1)$. These yield
$$
q\le \frac{f_Z(k)}{f_Z(k-1)}\le q^{-1},
$$
that is
$$
-pq^{-1}\le 1-\frac{f_Z(k-1)}{f_Z(k)}\le p.
$$
For all $t\in [-pq^{-1}, p]$, we have
\begin{align*}
(1+p^{-1}qt)^\alpha+(1-p^{-1}t)^\alpha &\ge (1+q)^\alpha \quad\quad\text{if}~0<\alpha<1\\
(1+p^{-1}qt)^\alpha+(1-p^{-1}t)^\alpha &\le (1+q^{-1})^\alpha \quad\text{if}~\alpha>1.
\end{align*}
Then we arrive at
\begin{align*}
\sum_{k\in\mathbb{Z}}f_Y(k)^\alpha &\ge \sum_{k=1}^\infty((1+q)f_Z(k))^\alpha \quad\quad\text{if}~0<\alpha<1\\
\sum_{k\in\mathbb{Z}}f_Y(k)^\alpha  &\le \sum_{k=1}^\infty((1+q^{-1})f_Z(k))^\alpha \quad\text{if}~\alpha>1.
\end{align*}
Neither $\{(1+q)f_Y(k)\}_{k=1}^\infty$ nor $\{(1+q^{-1})f_Y(k)\}_{k=1}^\infty$ are the appropriate intermediate functions. This is the reason for the intermediate probability mass function $\{\phi(k)\}_{k=1}^\infty$ being defined as $\phi(k)=f_Y(k)+f_Y(1-k)$. 
\end{rmk}

\begin{thm}\label{thm:symm-geo-var}
Let $0<p<1$ and write $q:=1-p$. Let $X\sim\text{Geom}(p)$ be a geometric random variable with the probability mass function $f_X(k)=pq^k$ for $k\in \mathbb{Z}_{\ge 0}$. Let $Y$ be an integer-valued independent symmetrizer of $X$. Then we have
$$
\Var(Y)\ge \Var(X).
$$
Equality holds if and only if $Y$ is an independent copy of $-X$.
\end{thm}

\begin{proof}
The theorem trivially holds if $\Var(Y)=\infty$. Thus we will assume that $\Var(Y)<\infty$. Write $Z=X+Y$. Since $X$ and $Y$ are independent, we have 
$$
\Var(Z)=\Var(X)+\Var(Y).
$$
Therefore, it would suffice to show that
\begin{equation}\label{eq:var-z-x}
\Var(Z)\ge 2\Var(X). 
\end{equation}
Let $W$ be a random variable with probability mass function $\{\phi(k)\}_{k=1}^\infty$ as defined in \eqref{eq:phi}. Clearly, $\P(W\ge k)=\sum_{j=k}^\infty \phi(j)=R_k$. Then we rewrite inequality \eqref{eq:T-k-k+1} as
$$
\P(W\ge k+1)\ge q\P(W\ge k),
$$
which, by induction, gives for $k\ge 0$ that
$$
\P(W\ge k+1)\ge q^k.
$$
Note that $\P(X\ge k)=q^k$ for $k\ge 0$. Thus $W-1$ has first-order stochastic dominance over $X$. By the symmetry $f_Z(j)=f_Z(-j)$, we have
\begin{align*}
\E [W(W-1)]&=\sum_{k=1}^\infty k(k-1)\phi(k)\notag\\
&=\sum_{k=1}^\infty k(k-1)(f_Z(k)+f_Z(k-1)) \notag\\
&=\sum_{k=1}^\infty k(k-1)f_Z(k)+\sum_{k=0}^\infty (k+1)kf_Z(k)\notag\\
&=2\sum_{k=1}^\infty k^2f_Z(k)=\sum_{k\in\mathbb{Z}}^\infty k^2f_Z(k)=\E Z^2=\Var (Z).
\end{align*}
Since $W-1$ has first-order stochastic dominance over $X$ and $n(n+1)$ is increasing on $\mathbb{Z}_{\ge 0}$,  we have that
$$
\E [W(W-1)]\ge \E[X(X+1)]=\frac{2q}{p^2}=2\Var(X).
$$
Combining the above estimates of $\E[W(W-1)]$, we obtain the desired inequality \eqref{eq:var-z-x}. 

\textbf{Equality case.} As the proof reveals, the equation $\Var (Y)=\Var (X)$ forces that $W-1$ and $X$ have the same distribution. Therefore we obtain 
\begin{align*}
\phi(k) &=\mathbb{P}(W\ge k)-\mathbb{P}(W\ge k+1)\\
&=\mathbb{P}(X\ge k-1)-\mathbb{P}(X\ge k)\\
&=pq^{k-1}. 
\end{align*}
As discussed in the equality case of Theorem \ref{thm:symm-geo-ent}, we can then solve equation \eqref{eq:phi} to obtain $f_Z$, which together with the inversion formula \eqref{eq:pmf-y}, will give $f_Y$. Then one can see that $Y$ has the same distribution as $-X$. 
\end{proof}

We next establish a concentration inequality analogous to Theorem \ref{thm:concentration-exp} for geometric random variables.

\begin{lem}\label{lem:discrete-hazard}
Let $0<p<1$. Let $\{\phi(k)\}_{k\ge1}$ be a probability mass function. Set $R_k=\sum_{j=k}^{\infty}\phi(j)$.
Suppose that $\phi(k)\le p R_k$ for all $k\ge1$. Then, for every finite set $B\subset\Z_{\ge1}$, we have
$$
\sum_{k\in B}\phi(k)\le 1-q^{|B|}.
$$
\end{lem}

\begin{proof}
There is nothing to prove when $B=\emptyset$. Otherwise, we write
$B=\{b_1<\cdots<b_m\}$ and define
$$
U_0=1,\quad U_k=1-\sum_{j=1}^{k}\phi(b_j) \quad \text{for}~ 1\le k\le m.
$$
Since $b_1,\ldots,b_{k-1}<b_k$, we have
$$
R_{b_k}=1-\sum_{i<b_k}\phi(i)\le 1-\sum_{j<k}\phi(b_j)=U_{k-1}.
$$
Consequently, we obtain
$$
\phi(b_k)\le pR_{b_k}\le pU_{k-1},
\quad U_k=U_{k-1}-\phi(b_k)\ge qU_{k-1}.
$$
Iteration gives $U_m\ge q^m$, which is the asserted inequality.
\end{proof}

\begin{thm}\label{thm:concentration-geo}
Let $0<p<1$ and write $q:=1-p$. Let $X\sim\text{Geom}(p)$ be a geometric random variable with the probability mass function $f_X(k)=pq^k$ for $k\in \mathbb{Z}_{\ge 0}$. Let $Y$ be an integer-valued independent symmetrizer of $X$. Then, for every finite set $A\subset\Z$, we have
\begin{equation}\label{eq:concentration-geo}
\P(Y\in A)\le 1-q^{|A|}.
\end{equation}
The estimate is sharp for every prescribed cardinality and equality can be achieved by
$Y\stackrel{d}{=}-X$
and 
$A=\{1-|A|, \cdots, -1, 0\}$.
\end{thm}

\begin{proof}
Let $f_Y$ and $f_Z$ denote the probability mass functions of $Y$ and $Z$, respectively. Decompose $A=A_+\cup A_-$, where $A_+=A\cap \Z_{\ge 1}$ and $A_-=A\cap \Z_{\le 0}$. Then we have
\begin{align}\label{eq:prob y-in-a-geo}
\P(Y\in A) &=\sum_{k\in A_+}f_Y(k)+\sum_{k\in A_-}f_Y(k) \notag\\
&=\sum_{k\in A_+}f_Y(k)+\sum_{k\in 1-A_-}f_Y(1-k)\notag\\
&\le\sum_{k\in A_+\cup (1-A_-)}(f_Y(k)+f_Y(1-k))\notag\\
&=\sum_{k\in B}\phi(k).
\end{align}
Here, $B=A_+\cup (1-A_-)$ and $\{\phi(k)\}_{k=1}^\infty$ is a probability mass function defined in \eqref{eq:phi} and \eqref{eq:phi-1} as
$$
\phi(k)=f_Z(k-1)+f_Z(k)=f_Y(1-k)+f_Y(k).
$$
We have proved in \eqref{eq:phi-k-k+1} that $\phi(k+1)\ge q\phi(k)$, which gives $\phi(k+r)\ge q^r\phi(k)$ for all $r\ge0$, and thus
$$
R_k=\sum_{j=k}^{\infty}\phi(j)\ge\sum_{r=0}^{\infty}q^r\phi(k)=\frac{\phi(k)}p.
$$
Then we apply Lemma~\ref{lem:discrete-hazard} to obtain 
$$
\sum_{k\in B}\phi(k)\le 1-q^{|B|}\le 1-q^{|A|}.
$$
Combined with \eqref{eq:prob y-in-a-geo}, this gives the concentration bound \eqref{eq:concentration-geo}.
\end{proof}

\begin{rmk}
Analogous to Remark \ref{rmk:concentration}, concentration inequality \eqref{eq:concentration-geo} yields the majorization of $X$ over $Y$, and this implies the entropic comparisons \eqref{eq:geo-renyi-y-x} and \eqref{eq:geo-tsallis-y-x}.
\end{rmk}


{\bf Acknowledgement.} The problem of symmetrizing exponential random variables was communicated to us by Mokshay Madiman. This work is supported by the National Natural Science Foundation of China (NSFC) grant 62201175.


\end{document}